\documentclass[a4paper, 10pt, oneside, onecolumn]{article}

\RequirePackage{geometry}
\usepackage[utf8]{inputenc}
\usepackage[T1]{fontenc}
\usepackage{lmodern}
\usepackage{amsmath}
\usepackage{amssymb}
\usepackage{amsfonts}
\usepackage{amsthm}
\usepackage{xcolor}
\usepackage{hyperref}
\usepackage{cleveref}
\usepackage{mathtools}
\usepackage[autostyle=true]{csquotes}

\hypersetup{
	colorlinks=false,
	pdfborder={0 0 0},
	pdftitle={Generalized solutions for a system of partial differential equations arising from urban crime modeling with a logistic source term},
	pdfauthor={Frederic Heihoff},
	pdfkeywords={},
	bookmarksopen=true,
}

\renewcommand{\phi}{\varphi}

\newtheorem{base}{Base}[section]
\numberwithin{equation}{section}

\theoremstyle{plain}
\newtheorem{theorem}[base]{Theorem}
\newtheorem{lemma}[base]{Lemma}
\newtheorem{prop}[base]{Proposition}

\theoremstyle{definition}
\newtheorem{definition}[base]{Definition}

\newcommand*{\medcap}{\mathbin{\scalebox{1.5}{\ensuremath{\cap}}}}

\newcommand{\R}{\mathbb{R}}
\newcommand{\N}{\mathbb{N}}
\renewcommand{\d}{\,\mathrm{d}}
\newcommand{\laplace}{\Delta}
\newcommand{\grad}{\nabla}
\renewcommand{\div}{\nabla \cdot}
\renewcommand{\L}[1]{{L^{#1}(\Omega)}}
\newcommand{\defs}{\coloneqq}
\newcommand{\stext}[1]{\;\;\text{ #1 }\;\;}
\newcommand{\eps}{\varepsilon}
\newcommand{\loc}{\mathrm{loc}}
\newcommand{\tmax}{{T_{\mathrm{max}}}}
\newcommand{\tmaxeps}{T_{\mathrm{max}, \eps}}
\newcommand{\supp}{\text{supp}}

\newcommand\numberthis{\addtocounter{equation}{1}\tag{\theequation}}

\makeatletter
\g@addto@macro\bfseries{\boldmath}
\makeatother

\title{Generalized solutions for a system of partial differential equations arising from urban crime modeling with a logistic source term}
\author{
	Frederic Heihoff\footnote{fheihoff@math.uni-paderborn.de}\\
	{\small Institut f\"ur Mathematik, Universit\"at Paderborn,}\\
	{\small 33098 Paderborn, Germany}
}
\date{}

\begin{document}
\maketitle
\begin{abstract}
	\noindent We consider the system
	\[
		\left\{
		\begin{aligned}
		u_t &= \Delta u - \chi \nabla \cdot ( \tfrac{u}{v}  \nabla v) - uv + \rho u - \mu u^2, \\
		v_t &= \Delta v - v + u v
		\end{aligned} 
		\right.
		\tag{$\star$}		
	\]
	with $\rho \in \mathbb{R}, \mu > 0, \chi > 0$ in a bounded domain $\Omega \subseteq \mathbb{R}^2$ with smooth boundary. While very similar to chemotaxis models from biology, this system is in fact inspired by recent modeling approaches in criminology to analyze the formation of crime hot spots in cities. The key addition here in comparison to similar models is the logistic source term. 
	\\[0.5em]
    The central complication this system then presents us with, apart from us allowing for arbitrary $\chi > 0$, is the nonlinear growth term $uv$ in the second equation as it makes obtaining a priori information for $v$ rather difficult. Fortunately, it is somewhat tempered by its negative counterpart and the logistic source term in the first equation. It is this interplay that still gives us enough access to a priori information to achieve the main result of this paper, namely the construction of certain generalized solutions to ($\star$).
	\\[0.5em]
	To illustrate how close the interaction of the $uv$ term in the second equation and the $-\mu u^2$ term in the first equation is to granting us classical global solvability, we further give a short argument showing that strengthening the  $-\mu u^2$ term to $-\mu u^{2+\gamma}$ with $\gamma > 0$ in the first equation directly leads to global classical solutions.
	\\[0.25em]
\end{abstract}

\section{Introduction}
In this paper we discuss the system 
\begin{equation}
\left\{
\begin{aligned}
u_t &= \laplace u - \chi \div ( \tfrac{u}{v}  \grad v) - uv + \rho u - \mu u^2,   \\
v_t &= \laplace v - v + u v
\end{aligned}
\right. \label{problem}
\end{equation}
with $\rho \in \R$, $\mu > 0$, $\chi > 0$. While this system is in fact motivated by recent modeling approaches in criminology, we will first establish some of its broader context and therefore take a quick detour to biology and the mathematical modeling of chemotactic movement of certain microscopic organisms. Here, chemotaxis means the process whereby organisms move along a chemical gradient towards an attractant. Modeling this process using systems of partial differential equations has largely been started by the seminal work of Keller and Segel in 1970  (cf.\ \cite{keller1970initiation}), in which they modeled a population of \enquote*{dictyostelium discoideum} slime mold to understand their aggregation behavior observed in experiments using the following (here somewhat simplified) system:
\begin{align*}
\left\{
	\begin{aligned}
	u_t &= \laplace u - \div (u\grad v) \\
	v_t &= \laplace v - v + u
	\end{aligned}
\right.
\end{align*}
In this system, the functions $u$ and $v$ model the cell and attractant concentrations while the term $-\div (u\grad v)$ models the central mechanism of the model, namely the chemotaxis. The remaining terms either model the diffusion of cells or attractant or their production and decay behavior.
\\[0.5em]
The efficacy of this approach was then later confirmed from a mathematical perspective as the same aggregation behavior is also present in solutions to this system for a large set of initial data when considered in three dimensions. In mathematical terms, this is expressed by solutions blowing up in finite time but conserving their mass (cf.\ \cite{MR3115832}). Among others, this success has then led to various further chemotactic processes from biology being modeled and then subsequently mathematically analyzed in recent years. For a broader overview of this, see \cite{MR3351175}.  
\\[0.5em]
As biology is not the only field concerned with analyzing the movement of agents towards some kind of goal, other fields have taken notice of this new, successful modeling approach and translated it to their setting. One such field is criminology. Here, cells are replaced by criminals and attractant chemicals are replaced by a somewhat more abstract notion of attractiveness of locations for criminal activity. As such with the central goal of understanding crime hot spots, Short et al. introduced the following (here somewhat simplified) system in 2008 (cf.\ \cite{short2008statistical}), which is based on a \enquote*{routine activity} modeling approach (cf.\ \cite{cohen1979social} and \cite{felson1987routine}) and insights gained in \cite{johnson1997new}, \cite{short2009measuring} and \cite{wilson1982broken} about repeat victimization and crime and disorder generally leading to more of the same:
\begin{equation}
\left\{
\begin{aligned}
u_t &= \laplace u - \chi\div (\tfrac{u}{v}\grad v) - uv + \Psi \\
v_t &= \laplace v - v + uv + \Phi
\end{aligned}
\right.
\label{classic_crime_model}
\end{equation} 
with $\chi = 2$.
In this system $u$ and $v$ represent the criminal population and attractiveness factor for criminal activity respectively, while $-\chi\div (\tfrac{u}{v}\grad v)$ is still the, here slightly modified, taxis term, which this time models the tendency of criminals to move towards high attractiveness areas. This modification to the taxis term has been introduced by Short et al. to account for the fact that there is an aspect of diminishing returns to consider regarding attractiveness, meaning that a high attractiveness of the current location of a criminal makes them much less likely to move from there as even a higher attractiveness areas do not seem much better in comparison. The paired $uv$ terms are meant to represent expected values of crime in an area at a certain time modeling essentially that crime in an area leads to higher attractiveness and less repeat crime (cf.\ again \cite{johnson1997new}, \cite{short2009measuring} and \cite{wilson1982broken}). The functions $\Phi$ and $\Psi$ further represent some growth information about criminals and attractiveness independent of the model functions $u$ and $v$, e.g.\ the socio-economic state of certain areas of a city at certain points in time influencing criminalization and creation of attractive targets for criminal activity. For a broader survey of models derived from this, see \cite{d2015statistical}.
\\[0.5em]
In terms of the mathematical analysis of this model, there have been e.g.\ global classical existence results in one dimension in \cite{rodriguez2019global} and arbitrary dimension, but with some restrictions on $\chi$, in \cite{MR3879245}. Furthermore, existence of solutions for the two-dimensional, radially symmetrical case has been studied in \cite{MR4002172} and a similar result for classical solutions given small initial data can be found in \cite{PreprintSmallDataSolution}. See also \cite{PreprintNonLinearDiffusion}, in which existence of certain weak solutions for a variant of (\ref{problem}) with sufficiently strong nonlinear diffusion is discussed. As it is the central feature of interest from an application perspective, there have also been various discussions of hot spot formation in e.g.\  \cite{MR3163243}, \cite{MR2982715}, \cite{MR3491512}. For some theory about models from biology featuring a similar singular sensitivity function in various settings see e.g.\  \cite{MR2778870}, \cite{MR3674184}. \cite{1937-1632_2020_2_119} or \cite{MR3887138} for a case also featuring a logistic source term.
\\[0.5em]
Let us now return our focus to the model (\ref{problem}), which is the central object of study in this paper. While it is still very similar to the classic model  (\ref{classic_crime_model}) introduced by Short et al., there exist some important differences, namely that we removed the static source terms $\Phi$ and $\Psi$ for convenience of notation, but introduced an additional logistic source term in the first equation. Source terms of this kind are a fairly standard addition to chemotaxis models in biology to represent that cells reproduce while still incorporating the idea that this reproduction even when considered in isolation cannot be unbounded as cells compete for some finite resources, e.g.\ space. This idea then fairly cleanly translates to criminals, where reproduction is replaced by criminalization of individuals in the area by the existing criminal population while criminals still compete with each other for e.g.\  good targets, which are a limited resource.
\paragraph{Main result.}
The main result of this paper is the construction of certain generalized solutions for the system (\ref{problem}) similar to those considered in \cite{MR3383312} or \cite{MR3859449}. Or put more precisely, we consider the following setting: We study the system (\ref{problem}) with parameters $\rho \in \R$, $\mu > 0$, $\chi > 0$ in a bounded domain $\Omega \subseteq \R^2$ with smooth boundary. We further add the boundary conditions 
\begin{equation}
	\grad u \cdot \nu = \chi\tfrac{u}{v} \grad v \cdot \nu, \;\;\;\; \grad v \cdot \nu = 0 \;\;\;\; \text{ for all } x\in\partial\Omega, t > 0
\label{boundary_conditions}
\end{equation}
and initial conditions 
\begin{equation}
	u(x,0) = u_0(x), \;\;\;\; v(x,0) = v_0(x) \;\;\;\; \text{ for all } x \in \Omega
	\label{intial_conditions}
\end{equation}
for initial data with the following properties:
\begin{equation}
	\left\{\;
	\begin{aligned}
		&u_0 \in C^0(\overline{\Omega}) \;\;\;\;\;\;&& \text{with } u_0 \geq 0  \,\text{ in } \overline{\Omega}, \\
		&v_0 \in W^{1,\infty}(\Omega) && \text{with } v_0 > 0 \;\text{ in } \overline{\Omega}
	\end{aligned}
	\right. .
	\label{intial_data_regularity}
\end{equation}
For the sake of simplicity, we fix the domain $\Omega$ and parameters $\rho, \mu, \chi$ from here on out. 
\\[0.5em]
Under these assumptions, we then derive the following existence result:
\begin{theorem}
	\label{theorem:main}
	The system (\ref{problem}) with boundary conditions (\ref{boundary_conditions}) and initial data (\ref{intial_conditions}) with properties (\ref{intial_data_regularity}) has a global generalized solution $(u,v)$ in the sense of \Cref{definition:weak_solution} below.
\end{theorem}
\paragraph{Complications.} 
As is common to most (chemo)taxis type systems, the taxis term is always somewhat of a complication because it often stands in the way of easy access to a priori information for the first solution component. In our case, it could be argued that this is amplified by the fact that we allow it to be arbitrarily strong (meaning allowing for arbitrarily big values of $\chi > 0$, $\chi = 2$ being the critical case in the classic Short model (\ref{classic_crime_model})) and have it include a singular sensitivity function $\frac{u}{v}$. While the singular sensitivity might seem critical at first glance, it poses at least for existence theory only negligible problems. This is the case because at least for finite times there always exists a positive lower bound for $v$ by straightforward use of semigroup methods, which means that for all the relevant existence theory the sensitivity $\frac{u}{v}$ is no more problematic than a sensitivity of the form $u$. On the other hand, allowing for arbitrary $\chi > 0$ and therefore $\chi = 2$ seems to be much more of a hurdle to constructing solutions as it makes adapting the techniques seen in e.g.\ \cite{MR3879245} for small values of $\chi$ infeasible.
\\[0.5em]
Apart from the fairly standard complications introduced by the taxis term, the main complication in terms of us being able to derive sufficient a priori estimates to allow for the existence of global solutions is the nonlinear $uv$ term in the second equation. While it can be played against a similar term in the first equation to at least gain some initial $L^1$ type estimates for $u$ and $v$, it is still highly problematic when trying to derive higher $L^p$ bounds for the second solution component. This problem is only slightly tempered by the integrability properties for $\int_\Omega u^2$ granted to us by the logistic source term in the first equation of (\ref{problem}), which allow us to at least gain $L^p$ bounds for $v$ and any finite $p$, but are to our knowledge not quite enough to gain the critical $L^\infty$ bound for $v$ we would need to gain classical solutions. 

\paragraph{Existence of classical solutions given a stronger logistic source.} To illustrate how critical this interaction of the logistic source term in the first equation and the growth term $uv$ in the second equation is in two dimensions, we will in this paper also consider an altered version of (\ref{problem}) with a slightly strengthened logistic source term, namely
\begin{equation}
	\left\{
	\begin{aligned}
	u_t &= \laplace u - \chi \div ( \tfrac{u}{v}  \grad v) - uv + \rho u - \mu u^{2+\gamma},  \\
	v_t &= \laplace v - v + uv
	\end{aligned}
	\right. \label{weaker_problem}
\end{equation}
with $\gamma > 0$, which is also fixed from here on out similar to the other parameters. While this system is in fact very similar to (\ref{problem}), we will later see in \Cref{section:weakend_case} that this small addition of a slightly stronger logistic source term directly leads to classical solvability, or more precisely, to the following proposition: 
\begin{prop}
	\label{prop:weaker_system_stronger_solution}
	The system (\ref{weaker_problem}) with boundary conditions (\ref{boundary_conditions}) and initial data (\ref{intial_conditions}) with properties (\ref{intial_data_regularity}) has a unique, global classical solution $(u,v)$.
\end{prop}
\noindent This result is mostly made possible due to the fact that the stronger logistic source term allows us to bridge a critical gap in a priori information for $v$ (more precisely it lets us derive an $L^\infty$ bound for $v$ and some crucial bounds for the gradient of $v$ as seen in \Cref{lemma:weak_higher_v_bounds}). Considered in this way, our case therefore seems to be just on the boundary to classical solvability, but as far as we know only allows for e.g.\ generalized solutions in the sense of \Cref{definition:weak_solution}.

\paragraph{Approach.} While some ideas could maybe already be gleamed from the discussion of the critical terms in (\ref{problem}), let us now give a more detailed overview of our approach in this paper:
\\[0.5em]
As is common when constructing weak or generalized solutions, our approach is based on the analysis of regularized versions of the problem (\ref{problem}), indexed by $\eps \in (0,1)$ (cf.\ (\ref{approx_problem})), that admit global classical solutions $(u_\eps, v_\eps)$ and approach the original problem as $\eps \searrow 0$. It is then our aim to derive bounds for these approximate solutions independent of $\eps$ and use well-known compact embedding properties of certain function spaces (e.g.\ due to the Aubin--Lions lemma) to gain solution candidates as limits of the approximate solutions $(u_\eps, v_\eps)_{\eps \in (0,1)}$ along a suitable sequence $(\eps_j)_{j\in\N} \subseteq (0,1)$ with $\eps_j \searrow 0$ as $j \rightarrow \infty$. The last step is then to derive sufficient convergence properties for the sequence $(u_{\eps_j}, v_{\eps_j})_{j\in\N}$ to translate the necessary solutions properties from the approximate solutions to our solution candidates.
\\[0.5em]
The key point in this approach (as in many others) is the derivation of sufficient a priori information. While some baseline $L^1$ estimates can be gained by the fairly common approach to add the first two equations in (\ref{approx_problem}) to cancel out the $u_\eps v_\eps$ terms, it is higher $L^p$ bounds for $v$ and its gradients where the key insight in this paper comes in. To derive these, we first notice that the logistic source term in the first equation in (\ref{approx_problem}) gives us a very useful integrability property for $\int_\Omega u_\eps^2$, which can then be used when testing the second equation in (\ref{approx_problem}) with $v_\eps^{p-1}$ to rein in the problematic influence coming from the resulting $u_\eps v_\eps^p$ terms just about enough to gain $L^p$ bounds for $v_\eps$ for all finite $p$ and an integrability property for $\int_\Omega|\grad v_\eps|^2$ (cf.\ \Cref{lemma:v_bounds}) due to us only considering a two-dimensional setting. It is both of these properties that lead us to the necessary compact embedding properties for the second solution component and allow us to derive a useful integrability property for $\int_\Omega\frac{|\grad u_\eps|^2}{(u_\eps + 1)^2}$ by testing the first equation with $\frac{1}{u_\eps + 1}$. While the latter does not help us in deriving further a priori estimates for the first solution component itself, this integrability property at least ensures sufficient compact embedding properties for $\ln(u_\eps + 1)$. 
\\[0.5em]
Though the above a priori information already grants us most of the convergence properties we need to translate solution properties from the approximate solutions to the solution candidates as a consequence of the used compactness arguments, we devote \Cref{section:grad_v_convergence} to deriving some additional convergence properties for $\grad v_\eps$ by adapting methods found in e.g.\ \cite{MR3383312} and \cite{MR3859449}. These additional properties are mostly necessary to handle the taxis-induced terms. 

\section{Generalized solution concept and approximate solutions}
Due to the complications laid out in the introduction, classical solutions to (\ref{problem}) seem to us to be out of reach for now and as such we will in this paper focus on a more generalized solution concept similar to the one introduced in e.g.\ \cite{MR3383312}. These solutions are defined as follows:
\begin{definition}
	We call nonnegative functions $u,v$ with
	\begin{equation}
		\begin{aligned}
		u &\in L^2_\loc(\overline{\Omega}\times[0,\infty))\cap L^\infty([0,\infty);L^1(\Omega)) , \\
		\ln(u + 1) &\in L^{2}_\loc([0,\infty);W^{1,2}(\Omega)), \\
		v &\in \medcap_{p\geq 1} L_\loc^\infty([0,\infty); L^p(\Omega)) \cap L_\loc^2([0,\infty); W^{1,2}(\Omega)) \text{ and } \\
		v^{-1} &\in L^\infty_\loc(\overline{\Omega}\times[0,\infty))
		\end{aligned}
		\label{wsol:regularity}	
	\end{equation}
	a generalized solution of (\ref{problem}) with (\ref{boundary_conditions}) and (\ref{intial_conditions}), if 
	\[
		\int_\Omega u(\cdot, T) - \int_\Omega u_0 \leq -\int_0^T\int_\Omega uv + \rho\int_0^T \int_\Omega u - \mu \int_0^T \int_\Omega u^2
		\numberthis
		\label{wsol:mass_property}
	\]
	for a.e.\ $T > 0$ and
	\begin{align*}
		-\int_0^\infty \int_{\Omega} \ln(u+1)\varphi_t - \int_{\Omega} \ln(u_0 + 1)\varphi(\cdot, 0) \geq& \int_0^\infty \int_{\Omega} \ln(u+1)\laplace\varphi + \int_0^\infty \int_{\Omega} |\grad \ln(u+1)|^2\varphi \\
		& - \chi \int_0^\infty \int_{\Omega} \frac{u}{v(u+1)} (\grad \ln(u+1) \cdot \grad v)\varphi \\
		& + \chi\int_0^\infty \int_{\Omega} \frac{u}{v(u+1)} \grad v \cdot \grad \varphi \\
		& -\int_0^\infty \int_\Omega \frac{uv}{u+1}\phi \\
		& + \rho\int_0^\infty \int_\Omega \frac{u}{u+1}\phi - \mu \int_0^\infty \int_\Omega \frac{u^2}{u+1} \phi \numberthis \label{wsol:ln_u_inequality}
	\end{align*}
	holds for all nonnegative $\phi \in C^\infty_0(\overline{\Omega}\times[0,\infty))$ with $\grad \phi \cdot \nu = 0$ on $\partial \Omega\times(0,\infty)$ and if 
	\[
		\int_0^\infty\int_\Omega v \phi_t + \int_\Omega v_0 \phi(\cdot, t)  = \int_0^\infty \int_\Omega \grad v \cdot \grad \phi + \int_0^\infty\int_\Omega v \phi - \int_0^\infty\int_\Omega uv \phi
		\numberthis
		\label{wsol:v_inequality}
	\]
	holds for all $\phi \in\medcap_{p\geq 1} L^\infty((0,\infty); \L{p}) \cap  L^2((0,\infty); W^{1,2}(\Omega))$ with $\phi_t \in L^2(\Omega\times[0,\infty))$ and compact support in $\overline{\Omega}\times[0,\infty)$.
	\label{definition:weak_solution}
\end{definition}
\noindent First note that, due to the regularity properties in (\ref{wsol:regularity}), all the integrals in the above definition are well-defined.
\\[0.5em]
\noindent Let us now briefly argue that this solution concept is sensible, meaning that classical solutions of (\ref{problem}) with (\ref{boundary_conditions}) and (\ref{intial_conditions}) are generalized solutions and sufficiently regular generalized solutions are in fact classical. 
That classical solutions satisfy \Cref{definition:weak_solution} is fairly easy to see by testing the first equation in (\ref{problem}) with $1$ as well as $\frac{\phi}{u + 1}$ and second equation in (\ref{problem}) with $\phi$ for appropriate functions $\phi$, applying partial integration and rearranging somewhat. As such, we will not expand on this point, but rather focus on the opposite direction, which is far more tricky and non-obvious. We will therefore now give the full argument for this based on prior work in \cite[Lemma 2.1]{MR3383312} for a similar generalized solution concept.
\begin{lemma}
	If $u,v \in C^{2,1}(\overline{\Omega}\times(0,\infty)) \cap  C^0(\overline{\Omega}\times[0,\infty))$ is a generalized solution in the sense of \Cref{definition:weak_solution} with initial data according to (\ref{intial_data_regularity}), then it is already a classical solution of (\ref{problem}) with boundary conditions (\ref{boundary_conditions}) and initial conditions (\ref{intial_conditions}). 
\end{lemma}
\begin{proof}
	As (\ref{wsol:v_inequality}) is a fairly standard weak solution formulation for the second equation in (\ref{problem}) and therefore well-known arguments directly apply to show that $v$ is in fact a classical solution of said equation, we will focus our efforts here on the $u$ component and the two inequalities (\ref{wsol:mass_property}) and (\ref{wsol:ln_u_inequality}).
	\\[0.5em]
	As our first step, let us verify that $u$ satisfies its initial conditions. For this, we first fix a nonnegative $\psi \in C_0^\infty(\Omega)$ and a sequence of cut-off functions $(\zeta_i)_{i\in\N} \subseteq C_0^\infty([0,\infty))$ with 
	\[
		\zeta_i(t) \in [0,1] \;\; \text{ for all } t\in[0,\infty), \;\;\;\; \zeta_i(0) = 1, \;\;\;\; \supp(\zeta_i) \subseteq [0, \tfrac{1}{i}] \stext{ and } \zeta_i' \leq 0 \;\;\;\; \text{ for all }i\in\N
	\]
	in the same way as in \cite{MR3383312}.
	We then define $\phi_i(x,t) \defs \psi(x)\zeta_i(t)$ for all $i\in\N,x\in\Omega,t\in[0,\infty)$, which are of appropriate regularity to be test functions for (\ref{wsol:ln_u_inequality}). If we then plug these test functions into (\ref{wsol:ln_u_inequality}) and take the limit $i \rightarrow \infty$, we gain that
	\[
		\int_\Omega \ln(u(\cdot,0) + 1)\psi - \int_\Omega \ln(u_0 + 1)\psi \geq 0 \;\;\;\; \text{ for all } x \in \Omega
	\]
	due to the dominated convergence theorem and the sequence $(\zeta'_i)_{i\in\N}$ approaching the Dirac measure $-\delta(t)$. This and the fact that $\ln(\cdot + 1)$ is monotonically increasing then directly imply
	\[
		u(x,0) \geq u_0(x) \;\;\;\; \text{ for all } x \in \Omega. \numberthis \label{eq:u_initial_ineq}
	\]
	Because of (\ref{wsol:mass_property}) and the continuity of $u$, we further gain that
	\[
		\int_\Omega u(\cdot,0) \leq \int_\Omega u_0,
	\] 
	which then together with (\ref{eq:u_initial_ineq}) gives us
	\begin{equation}
		u(x,0) = u_0(x) \label{eq:initial_data_works}
	\end{equation}
	for all $x \in \Omega$.
	\\[0.5em]
	By reversing the partial integration steps that would lead from testing the first equation in (\ref{problem}) with $\frac{\phi}{u+1}$ to (\ref{wsol:ln_u_inequality}) and applying a straightforward density argument, we immediately see that $u$ satisfies
	\[
		\frac{u_t}{u+1} \geq \frac{\laplace u}{u + 1} - \frac{\chi \div (\frac{u}{v} \grad v)}{u+1} - \frac{uv}{u+1} + \rho \frac{u}{u+1} - \mu \frac{u^2}{u+1}
	\] 
	or 
	\[
		u_t \geq \laplace u - \chi \div \left(\tfrac{u}{v} \grad v\right) - uv + \rho u - \mu u^2 \numberthis \label{eq:comp1_ineq}
	\]
	after multiplication with $u+1 > 0$ on $\Omega\times(0,\infty)$. A similar density argument with $\phi$ supported near the boundary then further yields 
	\[
		\grad u \cdot \nu \geq \chi \tfrac{u}{v} \grad v \cdot \nu \numberthis \label{eq:grad_ienq}
	\]
	on $\partial\Omega\times(0,\infty)$ in a similar fashion. 
	\\[0.5em]
	Let us now assume that $u$ is not a classical solution of the first equation in (\ref{problem}). Because of continuity, there then exist open sets $U_1\subseteq \Omega,V_1 \subseteq [0,\infty)$ such that
	\[
		u_t > \laplace u -  \chi \div \left(\tfrac{u}{v} \grad v\right) - uv + \rho u - \mu u^2 \;\;\;\; \text{on } U_1\times V_1
	\]
	or open sets $U_2 \subseteq \partial \Omega, V_2 \subseteq [0,\infty)$ such that
	\[
		\grad u \cdot \nu > \chi \tfrac{u}{v} \grad v \cdot \nu \;\;\;\;\text{on } U_2\times V_2.
	\]
	or both.
	In the latter case, this combined with (\ref{eq:initial_data_works}) then implies that
	\begin{align*}
		\int_\Omega u(\cdot, T) - \int_\Omega u_0 = \int_0^T\int_\Omega u_t &\geq \int_0^T\int_{\partial\Omega} (\grad u \cdot \nu - \chi\tfrac{u}{v}\grad v \cdot \nu) - \int_0^T \int_\Omega uv + \rho\int_0^T \int_\Omega u - \mu \int_\Omega u^2 \\
		&> - \int_0^T \int_\Omega uv + \rho\int_0^T \int_\Omega u - \mu \int_\Omega u^2
	\end{align*}
	for all $T \in V_2$ after some partial integration steps, which contradicts (\ref{wsol:mass_property}). A similar contradiction can be derived for the remaining case and as such $u$ must solve the first equation in (\ref{problem}) classically. This completes the proof.
\end{proof}
\noindent After having now established that existence of these generalized solutions is in fact desirable, let us now proceed to laying the groundwork for their construction. 
To do this, we first fix a family of cut-off functions $(\eta_\eps)_{\eps \in (0,1)}$ with
\[
	\eta_\eps \in C_0^\infty([0,\infty)) \stext{such that} 0 \leq \eta_\eps \leq 1 \text{ in } [0,\infty) \stext{and} \eta_\eps \nearrow 1 \text{ pointwise in } [0,\infty) \text{ as } \eps \searrow 0.
\]
We then use these to define the following approximated and regularized version of (\ref{problem}) with (\ref{boundary_conditions}) and (\ref{intial_conditions}):
\begin{equation}
\left\{
\begin{aligned}
	{u_\eps}_t &= \laplace u_\eps - \chi\div ( \eta_\eps(u_\eps) \tfrac{u_\eps}{v_\eps}  \grad v_\eps) - u_\eps v_\eps + \rho u_\eps - \mu u_\eps^2,  && x \in \Omega, t > 0 \\
	{v_\eps}_t &= \laplace v_\eps - v_\eps + u_\eps v_\eps, && x \in \Omega, t > 0 \\
	\grad u_\eps \cdot \nu &= 0, \;\; \grad v_\eps \cdot \nu = 0, && x \in \partial\Omega, t > 0 \\
u_\eps(x,0) &= u_0(x), \;\;  v_\eps(x,0) = v_0(x), && x \in \overline{\Omega} \\ 
\end{aligned}
\right. \label{approx_problem} .
\end{equation}
This system or more precisely its solutions will play a key role in the construction of generalized solutions in the sense of \Cref{definition:weak_solution}. 
\\[0.5em]
As such, let us now briefly consider the changes made in (\ref{approx_problem}) as compared to (\ref{problem}), which, while small, do have substantial impact concerning the existence of global classical solutions to this system. This stems mostly from the fact that introducing a cut-off function into the taxis term allows us to gain a critical $L^\infty$ estimate for $u$ by straightforward comparison with a constant function. This is then enough to derive sufficient bounds to show that finite-time blow-up in all the necessary norms is impossible for a local solution gained by adaption of standard local existence theory. Let us now make this precise:
\begin{lemma}\label{lemma:approx_exist}
	For each $\eps \in (0,1)$ and initial data $(u_0, v_0)$ according to (\ref{intial_data_regularity}), there exist functions 
	\[
		u_\eps, v_\eps \in C^0(\overline{\Omega}\times[0,\infty))\cap C^{2,1}(\overline{\Omega}\times(0,\infty))
	\]
	such that \[
		u_\eps(x,t) \geq 0, \;\; v_\eps(x,t) \geq e^{-t} \inf_{y\in\Omega} v_0(y) > 0 \;\;\;\;\;\; \text{ for all } x \in \Omega, t \in [0, \infty) \numberthis \label{eq:approx_lower_bounds}
	\] and $(u_\eps, v_\eps)$ is a classical solution of (\ref{approx_problem}). 
\end{lemma}
\begin{proof}
	A standard contraction mapping argument adapted from e.g.\ \cite{MR2146345} immediately gives us a local solution of (\ref{approx_problem}) on $[0,\tmaxeps)$ for a maximal $\tmaxeps \in (0,\infty]$ and the following blow-up criterion:
	\begin{align*}
		&\text{If } \tmaxeps < \infty, \\
		&\text{then } \limsup_{t\nearrow \tmaxeps} \left\{ \, \|u_\eps(\cdot, t)\|_\L{\infty} + \|v_\eps(\cdot, t)\|_{W^{1,\infty}(\Omega)} \, \right\} = \infty \text{ or } \liminf_{t\nearrow \tmaxeps} \inf_{x\in\Omega} v_\eps(x,t) = 0
		\numberthis \label{eq:blow-up}
	\end{align*} 
	Nonnegativity of $u_\eps$ and $v_\eps$ then immediately follows by maximum principle. Further, by analyzing $v_\eps$ using its mild solution representation (relative to the semigroup $e^{t(\laplace - 1)}$), we see that
	\begin{equation}
		v_\eps(\cdot,t) = e^{t(\laplace - 1)}v_0 + \int_0^t e^{(t-s)(\laplace - 1)} u_\eps(\cdot,s) v_\eps(\cdot,s) \d s \geq e^{t(\laplace - 1)}v_0 \geq e^{-t}\inf_{x\in\Omega} v_0(x) \;\;\;\; \text{ for all } t \in [0,\tmaxeps). \label{eq:lower_bound_v}
	\end{equation}
	This lower bound for $v$ will not only be necessary for almost all of the following arguments, but also  implies the second property in (\ref{eq:approx_lower_bounds}) after we have demonstrated that $\tmaxeps = \infty$ later in this proof.
	\\[0.5em]
	To now show that finite-time blow-up is impossible, let us first assume the opposite, namely that $\tmaxeps < \infty$. The property (\ref{eq:lower_bound_v}) then immediately gives us a positive lower bound for $v_\eps$ on $\Omega\times[0,\tmaxeps)$, which already rules out one of the possible blow-up scenarios.
	As further $\eta_\eps(u)$ is zero and $u \mapsto \rho u - \mu u^2$ is negative for sufficiently big values of $u$, a standard comparison argument applied to the first equation in (\ref{approx_problem}) with an appropriate constant function gives us $K_1 > 0$ with 
	\[
		\|u_\eps(\cdot, t)\|_\L{\infty} \leq K_1 \;\;\;\; \text{ for all } t \in [0,\tmaxeps)
	\]
	making another blow-up scenario impossible.
	This then in turn allows us to use a similar comparison argument for the second equation in (\ref{approx_problem}) to gain $K_2 > 0$ such that
	\[
		\|v_\eps(\cdot, t)\|_\L{\infty} \leq K_2 \;\;\;\; \text{ for all } t \in [0,\tmaxeps)
	\]
	by comparing with the solution to the initial value problem $y'(t) = (K_1 - 1) y(t)$, $t \in [0,\tmaxeps)$, $y(0) = \|v_0\|_\L{\infty}$ extended in such a way as to be constant on each $\Omega\times\{t\}$ for all $t \in [0,\tmaxeps)$. Lastly, another use of the mild solution representation of $v_\eps$ in tandem with well-known smoothness estimates for the semigroup~$(e^{t\laplace})_{t\geq 0}$ gives us
	\[
		\|\grad v(\cdot, t)\|_\L{\infty} \leq K_3\|\grad v_0\|_\L{\infty} + K_3 \int_0^t (t-s)^{-\frac{1}{2}}e^{(-1 -\lambda)(t-s)} \|u\|_\L{\infty}\|v\|_\L{\infty} \leq K_4
	\]
	for all $t \in [0, \tmaxeps)$ and appropriate constants $K_3, K_4 > 0$. This rules out the final possible blow-up scenario in (\ref{eq:blow-up}), which implies that our assumption $\tmaxeps < \infty$ must have been wrong. As such, we have proven that $\tmaxeps = \infty$, which completes the proof.
\end{proof}
\noindent
For the rest of this paper, we now fix some initial data $(u_0, v_0)$ according to (\ref{intial_data_regularity}) and a corresponding family of approximate solutions $(u_\eps, v_\eps)_{\eps \in (0,1)}$ as constructed in \Cref{lemma:approx_exist}.

\section{A priori estimates}
\label{section:apriori}
This section will be mostly concerned with deriving a priori bounds for the approximate solutions that we fixed in the previous section as preparation for later convergence arguments. We start this process by combining the first two equations in (\ref{problem}) (because the $-u_\eps v_\eps$ in the first equation will cancel out its counterpart in the second equation) to gain some important baseline estimates:
\begin{lemma}
	\label{lemma:l1_l2_props}
	There exists $C_1 > 0$ such that
	\begin{equation}
		\|u_\eps(\cdot, t)\|_\L{1} \leq C_1, \;\; \|v_\eps(\cdot, t)\|_\L{1} \leq C_1 \;\;\;\; \text{ for all } t > 0 \label{eq:l1_bounded} 
	\end{equation}
	and, for each $T > 0$, there exists $C_2(T) > 0$ such that
	\begin{equation}
	    \int_0^T \int_\Omega u_\eps^2 \leq C_2(T)
		\label{eq:u_l2_weak_boundedness}
	\end{equation}
	for all $\eps \in (0,1)$.
\end{lemma}
\begin{proof}
	As our first step, we add the equations for $u$ and $v$ together and then integrate to gain that
	\begin{align*}
		\frac{\d }{\d t}\int_\Omega (u_\eps + v_\eps) &= -\int_\Omega v_\eps + \rho\int_\Omega u_\eps - \mu \int_\Omega u_\eps^2 \\
		&\leq -\int_\Omega v_\eps + |\rho|\int_\Omega u_\eps - \mu \int_\Omega u_\eps^2 \;\;\;\; \text{ for all } t > 0 \text{ and } \eps \in (0,1) \label{eq:l1_added}  \numberthis
	\end{align*}
	after partial integration and use of the boundary conditions. By Young's inequality, we see that
	\[
		\int_\Omega u_\eps \leq \frac{\mu}{|\rho| + 1}\int_\Omega u_\eps^2 + \frac{|\rho| + 1}{4\mu}|\Omega|
	\]
	or further that
	\[
		-\mu\int_\Omega u_\eps^2 \leq -(|\rho| + 1) \int_\Omega u_\eps + \frac{(|\rho| + 1)^2}{4\mu}|\Omega|
	\]
	for all $t > 0$ and $\eps \in (0,1)$. If we now apply this to (\ref{eq:l1_added}), we gain that
	\[
		\frac{\d }{\d t}\int_\Omega (u_\eps + v_\eps) \leq - \int_\Omega (u_\eps + v_\eps) + \frac{(|\rho| + 1)^2}{4\mu}|\Omega|
	\]
	for all $t > 0$ and $\eps \in (0,1)$. This immediately implies (\ref{eq:l1_bounded}) by a straightforward comparison argument with the constant 
	\[
		C_1 \defs \max\left(  \frac{(|\rho| + 1)^2}{4\mu}|\Omega|, \int_\Omega (u_0 + v_0) \right).
	\] 
	If we now slightly rearrange (\ref{eq:l1_added}) and integrate, we further see for each $T > 0$ that
	\[
		\mu \int_0^T \int_\Omega u_\eps^2 \leq \int_\Omega (u_0 + v_0) + T|\rho| C_1
	\]
	for all $\eps \in (0,1)$, which then directly implies our second result (\ref{eq:u_l2_weak_boundedness}) because $\mu > 0$.
\end{proof}
\noindent Having now established the baseline estimates above, we can proceed to deriving the linchpin for our main existence result of this paper, namely $L^p$ bounds for $v_\eps$ and an integrability property for terms of the form $\int_\Omega|\grad v_\eps^{p/2}|^2$ for all finite $p$. The argument used for this mainly rests on the integrability property for $\int_\Omega u_\eps^2$ afforded to us by the logistic source term in the first equation in (\ref{approx_problem}) and gained in the previous lemma. As such, the following lemma presents the key insight in this paper of how to use the logistic source term in the first equation to temper the influence of the $u_\eps v_\eps$ growth term in the second equation and gain just about enough a priori estimates for the construction of generalized solutions in a two-dimensional setting.
\begin{lemma}\label{lemma:v_bounds}
	For each $T > 0$ and $p > 1$, there exists $C(T,p) > 0$ such that
	\[
		\|v_\eps(\cdot, t)\|_\L{p} \leq C(T,p) \;\;\;\; \text{ for all } t \in (0,T]
	\]
	and
	\[
		\int_0^T \int_\Omega |\grad v_\eps^\frac{p}{2}|^2 \leq C(T,p)
	\]
	for all $\eps \in (0,1)$. 
\end{lemma}
\begin{proof}
	Fix $p > 1$. Using the well-known Gagliardo--Nirenberg inequality, we can then fix $K_1 > 0$ such that
	\[
		\|\phi\|^2_\L{4} \leq K_1 \|\grad \phi\|_\L{2}\|\phi\|_\L{2} + K_1\|\phi\|^2_\L{\frac{2}{p}}
	\]
	for all $\phi \in C^1(\Omega)$. Now testing the second equation in (\ref{approx_problem}) with $v_\eps^{p-1}$ and applying the above inequality results in
	\begin{align*}
		\frac{1}{p} \frac{\d }{\d t} \int_\Omega v_\eps^p =& -\frac{4(p-1)}{p^2}\int_\Omega |\grad v_\eps^\frac{p}{2}|^2 - \int_\Omega v_\eps^p + \int_\Omega u_\eps v_\eps^p \\
		\leq& -\frac{4(p-1)}{p^2}\int_\Omega |\grad v_\eps^\frac{p}{2}|^2 + \|u_\eps\|_\L{2}\|v_\eps^\frac{p}{2}\|^2_\L{4} \\
		\leq& -\frac{4(p-1)}{p^2}\int_\Omega |\grad v_\eps^\frac{p}{2}|^2 + K_1\|u_\eps\|_\L{2}\|\grad v_\eps^\frac{p}{2}\|_\L{2}\|v_\eps^\frac{p}{2}\|_\L{2} + K_1\|u_\eps\|_\L{2}\|v_\eps^\frac{p}{2}\|^2_\L{\frac{2}{p}} \numberthis \label{eq:vp_diff_ineq_proto}
	\end{align*}
	for all $t > 0$ and $\eps \in (0,1)$.
	Because of \Cref{lemma:l1_l2_props}, there exists a constant $K_2 > 0$ such that
	\[
		\|v^\frac{p}{2}_\eps\|^2_\L{\frac{2}{p}} = \left\{\int_\Omega v_\eps \right\}^p \leq K_2^p \;\;\;\;  \text{ for all } t > 0 \text{ and } \eps \in (0,1)
	\]
	and therefore we can improve (\ref{eq:vp_diff_ineq_proto}) using Young's inequality as follows:
	\begin{align}
		\frac{1}{p} \frac{\d }{\d t} \int_\Omega v_\eps^p \leq&-\frac{p-1}{p^2}\int_\Omega |\grad v_\eps^\frac{p}{2}|^2  + K_3\left\{ \int_\Omega u_\eps^2 \right\} \int_\Omega v_\eps^p + K_4\left(1+\int_\Omega u_\eps^2\right) \label{eq:vp_diff_ineq} 
	\end{align}
	for all $t > 0$ and $\eps \in (0,1)$ and with $K_3 \defs \frac{K_1^2 p^2}{12(p-1)}$ and $K_4 \defs \frac{1}{2} K_1 K_2^p$. 
	\\[0.5em]
	We now fix $T > 0$. Because we then know from \Cref{lemma:l1_l2_props} that there exists $K_5 > 0$ such that 
	\[
		\int_0^T \int_\Omega u_\eps^2 \leq K_5
	\]
	for all $\eps \in (0,1)$, integration of (\ref{eq:vp_diff_ineq}) gives us 
	\[
		\int_\Omega v_\eps^p(\cdot, t) + \frac{p-1}{p}\int_0^t\int_\Omega |\grad v_\eps^\frac{p}{2}|^2  \leq K_6 + p K_3 \int_0^t \left\{ \int_\Omega u_\eps^2 \right\} \int_\Omega  v_\eps^p\label{eq:vp_int_ineq} \numberthis 
	\]
	with $K_6 \defs pK_4(T + K_5) + \int_\Omega v_0^p$ for all $t\in[0,T]$ and $\eps \in (0,1)$. This then implies
	\[
		\int_\Omega v_\eps^p(\cdot, t) \leq K_6 \exp\left( pK_3 \int_0^t \int_\Omega u_\eps^2  \right) \leq K_6e^{pK_3 K_5} =: K_7
	\]
	for all $t \in (0,T]$ and $\eps \in (0,1)$ by Gronwall's inequality,
	which is our first desired result. By now combining this new $L^p$ bound for $v_\eps$ with (\ref{eq:vp_int_ineq}), we then further see that
	\[
		\int_0^T\int_\Omega |\grad v_\eps^\frac{p}{2}|^2 \leq \frac{p}{p - 1} \left[ K_6 + pK_3 K_5 K_7 \right]
	\]
	for all $\eps \in (0,1)$.
\end{proof}
\noindent As the $L^p$ bounds established for $v_\eps$ in the above lemma seem to be insufficient to gain higher $L^p$ bounds for $u_\eps$ or its derivatives due to the taxis term in the first equation of (\ref{approx_problem}), we will instead restrict ourselves to establishing bounds for $\ln(u_\eps + 1)$ and its derivatives as is not uncommon for this type of problem. Given that the $L^1$ bound for $u_\eps$ found in \Cref{lemma:l1_l2_props} already gives us all possible $L^p$ bounds with finite $p$ for $\ln(u_\eps + 1)$, we will focus in the following lemma on establishing a useful, albeit fairly weak bound for the first derivatives of $\ln(u_\eps + 1)$. This is mostly made possible by the integrability properties for $\int_\Omega|\grad v_\eps|^2$ and baseline estimates for $u_\eps$ and $v_\eps$ already derived in this section.
\begin{lemma}\label{lemma:grad_ln_u_bound}
	For each $T > 0$, there exists $C(T) > 0$ such that
	\[
		\int_0^T\int_\Omega \frac{|\grad u_\eps|^2}{(u_\eps + 1)^2} \leq C(T)
	\]
	for all $\eps \in (0,1)$. 
\end{lemma}
\begin{proof}
We start by fixing $T > 0$. We then test the first equation in (\ref{approx_problem}) with $\frac{1}{u_\eps+1}$ to see that
\begin{align*}
	\frac{\d}{\d t}\int_\Omega \ln(u_\eps + 1) &= \int_\Omega \frac{|\grad u_\eps|^2}{(u_\eps + 1)^2} - \chi\int_\Omega \frac{u_\eps}{v_\eps(u_\eps + 1)^2} \grad u_\eps \cdot \grad v_\eps - \int_\Omega \frac{u_\eps v_\eps}{u_\eps + 1} + \rho \int_\Omega \frac{u_\eps}{u_\eps + 1} - \mu \int_\Omega\frac{u^2_\eps}{u_\eps + 1} \\
	&\geq \frac{1}{2}\int_\Omega \frac{|\grad u_\eps|^2}{(u_\eps + 1)^2} - K_1 \int_\Omega |\grad v_\eps|^2 - \int_\Omega v_\eps - \mu \int_\Omega u_\eps \numberthis \label{eq:ln_u_test_1}
\end{align*}
for all $t \in [0,T]$ and $\eps \in (0,1)$ with $K_1 \defs \tfrac{\chi^2}{2}(\inf_{x\in\Omega} v_0(x))^{-2} e^{2T}$. Because \Cref{lemma:l1_l2_props} and \Cref{lemma:v_bounds} then give us a constant $K_2 > 0$ such that
\[
	\int_\Omega u_\eps(\cdot, t) \leq K_2, \;\; \int_\Omega v_\eps(\cdot, t) \leq K_2 \;\;\text{and}\;\; \int_0^T\int_\Omega |\grad v_\eps|^2 \leq K_2
\]
for all $t\in[0,T]$ and $\eps \in (0,1)$,
time integration and some rearranging of inequality (\ref{eq:ln_u_test_1}) results in
\[
	\int_0^T \int_\Omega \frac{|\grad u_\eps|^2}{(u_\eps + 1)^2} \leq 2\left[ \int_\Omega \ln(u_\eps(\cdot, T) + 1) + K_1 K_2 + (1+\mu)K_2T\right]  \leq 2 (K_1 + 1)K_2 + 2(1 + \mu) K_2 T
\]
for all $\eps \in (0,1)$. This completes the proof.
\end{proof}
\section{Construction of limit functions as solution candidates}
This section will now be focused on using the a priori bounds above to construct a sequence $(\eps_j)_{j\in\N}$, along which our approximate solutions converge towards some limit functions $u$, $v$. These will then later play the role of candidates to be a generalized solution in the sense of \Cref{definition:weak_solution}. As it is often the case, the construction of said sequence will be built on well-known compact embedding properties of various function spaces, chief among them those afforded to us by the Aubin--Lions lemma (cf.\ \cite{TemamNavierStokes}). Specifically to enable us to use said lemma, we will now derive the following integrability properties for the time derivatives of the families $(\ln(u_\eps + 1))_{\eps \in (0,1)}$ and $(v_\eps)_{\eps \in (0,1)}$:
\begin{lemma}\label{lemma:dt_bounds}
	For each $T > 0$, there exists a constant $C(T) > 0$ such that
	\[
		\int_0^T \|\partial_t \ln(u_\eps(\cdot, t) + 1)\|_{(W^{2,2}(\Omega))^\star} \d t \leq C(T) \numberthis \label{eq:dt_ln_u_ineq}
	\]
	and 
	\[
		\int_0^T \|{v_\eps}_t(\cdot, t)\|_{(W^{2,2}(\Omega))^\star} \d t \leq C(T)
		\numberthis
		\label{eq:dt_v_ineq}
	\]
	for all $\eps \in (0,1)$.
\end{lemma}
\begin{proof}
	To prove (\ref{eq:dt_ln_u_ineq}), we first fix $\phi \in W^{2,2}(\Omega)$ and then test the first equation in (\ref{approx_problem}) with $\frac{\phi}{u_\eps + 1}$ to gain that
	\begin{align*}
		\int_\Omega \partial_t \ln(u_\eps + 1) \phi =& \int_\Omega \frac{|\grad u_\eps|^2}{(u_\eps + 1)^2}\phi - \int_\Omega \frac{\grad u_\eps \cdot \grad \phi}{u_\eps + 1} - \chi\int_\Omega \frac{\eta_\eps(u_\eps) u_\eps }{v_\eps(u_\eps + 1)^2} (\grad v_\eps \cdot \grad u_\eps) \phi
		 \\
		&+\chi\int_\Omega \frac{\eta_\eps(u_\eps) u_\eps }{v_\eps(u_\eps + 1)} (\grad v_\eps \cdot \grad \phi)- \int_\Omega \frac{u_\eps v_\eps}{u_\eps + 1}\phi + \rho\int_\Omega \frac{u_\eps}{u_\eps + 1}\phi - \mu \int_\Omega \frac{u_\eps^2}{u_\eps + 1} \phi \numberthis \label{eq:u_phi_test} 
	\end{align*}
	for all $t \in [0,T]$ and $\eps \in (0,1)$. Due to Young's inequality, the Cauchy--Schwarz inequality and the fact that \Cref{lemma:approx_exist} gives us that
	\[
		\inf_{x\in\Omega} v_\eps(x, t) \geq e^{-t}\inf_{x\in\Omega} v_0(x) \geq e^{-T}\inf_{x\in\Omega} v_0(x) > 0
	\] for all $t \in [0,T]$ and $\eps \in (0,1)$, the above equality directly implies that there exists a constant $K_1(T) > 0$ such that
	\begin{align*}
		\left| \int_\Omega \partial_t \ln(u_\eps(\cdot, t) + 1) \phi \right| \leq K_1(T)\left(  \int_\Omega\frac{|\grad u_\eps|^2}{(u_\eps + 1)^2} + \int_\Omega |\grad v_\eps|^2 + \int_\Omega u_\eps + \int_\Omega v_\eps +  1 \right) \left\{ \|\phi\|_\L{\infty} +  \|\phi\|_{W^{2,2}(\Omega)} \right\}
	\end{align*}
	for all $t \in [0,T]$ and $\eps \in (0,1)$. Given now the boundedness and integrability properties in \Cref{lemma:l1_l2_props}, \Cref{lemma:v_bounds} and \Cref{lemma:grad_ln_u_bound} and the fact that $W^{2,2}(\Omega)$ embeds continuously into $L^\infty(\Omega)$, this directly implies the inequality (\ref{eq:dt_ln_u_ineq}).\\[0.5em]
	To now prove (\ref{eq:dt_v_ineq}), we again fix $\phi \in W^{2,2}(\Omega)$ and this time test the second equation in (\ref{approx_problem}) with $\phi$ to gain that
	\[
		\int_\Omega {v_\eps}_t \phi = - \int_\Omega \grad v_\eps \cdot \grad \phi - \int_\Omega v_\eps \phi + \int_\Omega u_\eps v_\eps \phi \numberthis \label{eq:v_phi_test} 
	\]
	for all $t \in [0,T]$ and $\eps \in (0,1)$. By similar reasoning as above, we can now find a constant $K_2(T) > 0$ such that
	\[
		\left|	\int_\Omega {v_\eps}_t \phi \right| \leq K_2(T) \left(
		\int_\Omega |\grad v_\eps|^2 + \int_\Omega u_\eps^2 +
		\int_\Omega v^2_\eps + 1
		\right)\left\{ \|\phi\|_\L{\infty} +  \|\phi\|_{W^{2,2}(\Omega)} \right\}
	\]
	for all $t \in [0,T]$ and $\eps \in (0,1)$ based on (\ref{eq:v_phi_test}). Again due to \Cref{lemma:l1_l2_props} and \Cref{lemma:v_bounds}, this implies (\ref{eq:dt_v_ineq}) and therefore completes the proof.
\end{proof}
\noindent With all of the preparations now firmly in place, we can use the Aubin--Lions lemma and Vitali's theorem to construct our solution candidates as the limits of our approximate solutions along a suitable sequence of $\eps \in (0,1)$. Apart from the extended convergence result presented in the sequel, we will also already derive most of the convergence properties needed to translate the necessary properties for a generalized solution from the approximate solutions to our solution candidates.
\begin{lemma}
	\label{lemma:subsequence_extraction}
	There exist a sequence $(\varepsilon_j)_{j\in\N} \subseteq (0,1)$ with $\varepsilon_j \searrow 0$ as $j\rightarrow\infty$ and a tuple $(u,v)$  of limit functions defined on $\Omega\times[0,\infty)$ such that
	\begin{equation}
	\left\{
	\begin{aligned}
	&u_\varepsilon \rightarrow u && \text{in } L^p_\loc(\overline{\Omega}\times[0,\infty)) \text{ for } p\in[1,2) \text{ and a.e.\ in } \Omega \times [0,\infty), \\
	&u_\varepsilon(\cdot, t) \rightarrow u(\cdot, t) && \text{in } L^p(\Omega) \text{ for } p \in [1,2) \text{ and a.e.\ } t > 0,  \\
	&u_\eps \rightharpoonup u \;\;\;\;\;\;\;\;\;\;\;\;\;\; && \text{in } L^2_\loc(\overline{\Omega}\times[0,\infty)), \\
	&\ln(u_\varepsilon + 1) \rightharpoonup \ln(u + 1) \;\;\;\;\;\;\;\;\;\;\;\;\;\; && \text{in } L^2_\loc([0,\infty);W^{1,2}(\Omega)), \\
	&v_\eps \rightarrow v && \text{in } L^p_\loc(\overline{\Omega}\times[0,\infty)) \text{ for all } p \geq 1 \text{ and a.e.\ in } \Omega \times [0,\infty), \\
	&v_\eps(\cdot, t) \rightarrow v(\cdot, t) && \text{in } L^p(\Omega) \text{ for } p \geq 1 \text{ and for a.e.\ } t > 0 \text{ and } \\
	&v_\eps  \rightharpoonup v && \text{in } L^2_\loc([0,\infty);W^{1,2}(\Omega)) 
	\end{aligned}
	\right. 
	\label{eq:basic_convergence_props}
	\end{equation}
	as $\eps = \eps_j \searrow 0$. Further, $u$ is nonnegative, $v$ has the property $v(x, t) \geq e^{-t}\inf_{y\in\Omega} v_0(y)$ for almost all $(x,t)\in\Omega\times[0,\infty)$  and both satisfy the regularity properties in (\ref{wsol:regularity}).
\end{lemma}
\begin{proof}
	As we will successively extract subsequences multiple times in this lemma, we always denote the latest considered sequence as $(\eps_j)_{j\in\N}$ for ease of notation and without loss of generality.
	\\[0.5em]
	Due to \Cref{lemma:l1_l2_props}, \Cref{lemma:v_bounds}, \Cref{lemma:grad_ln_u_bound}, and \Cref{lemma:dt_bounds} combined with the Aubin--Lions lemma (cf.\  \cite{TemamNavierStokes}), we immediately gain that the families $(\ln(u_\eps + 1))_{\eps \in (0,1)}$ and $(v_\eps)_{\eps \in (0,1)}$ are compact in $L^2_\loc([0,\infty);L^2(\Omega))$ and therefore in $L^2_\loc(\overline{\Omega}\times[0,\infty))$ with regard to the strong topology and in $L^2_\loc([0,\infty);W^{1,2}(\Omega))$ with regard to the weak topology. Due to the strong compactness above, successive extraction of subsequences then gives us a sequence $(\eps_j)_{j\in\N}$ converging to zero and limit functions $u, v$ with 
	\[
		\ln(u_\eps + 1) \rightarrow \ln(u + 1) \stext{and}	v_\eps \rightarrow v \stext{in} L^2_\loc(\overline{\Omega}\times[0,\infty)) \stext{as} \eps = \eps_j \searrow 0.
	\]
	Again by successive subsequence extraction, we gain that for the new subsequence $(\eps_j)_{j\in\N}$ the convergences above are additionally true in an almost everywhere pointwise sense. This directly implies that $u_\eps \rightarrow u$ almost everywhere pointwise as $\R \ni x \mapsto e^{x} - 1$ is continuous. Note here that it is these pointwise convergences that make sure that all the limit functions found in this lemma are identical. Using that the bounds in \Cref{lemma:l1_l2_props} imply that $(u_\eps)_{\eps \in (0,1)}$ is compact in $L^2_\loc(\overline{\Omega}\times[0,\infty))$ with regards to the weak topology and further using the other weak compactness properties mentioned above, we also directly gain the weak convergence properties posited in (\ref{eq:basic_convergence_props}) by more subsequence extraction arguments.
	\\[0.5em]
	To prove the remaining convergence properties, we will now heavily lean on the Vitali convergence theorem and the de La Vallée Poussin criterion for uniform integrability (cf.\ \cite[pp.\ 23-24]{PropabilitiesAndPotential}). To this end, let us first note that the almost everywhere convergence of $u_{\eps_j}$ to $u$ and $v_{\eps_j}$ to $v$ implies that, for almost every $t > 0$, $u_{\eps_j}(\cdot, t) \rightarrow u(\cdot, t)$ and $v_{\eps_j}(\cdot, t) \rightarrow v(\cdot, t)$ pointwise almost everywhere. We further know from previous observations in \Cref{lemma:l1_l2_props} and \Cref{lemma:v_bounds} that, for each $T > 0$, $p\in[1,\infty)$, there exists constants $K_1(T), K_2(T, p) > 0$ with
	\[
		\int_0^T \int_\Omega u_\eps^2 \leq K_1(T), \;\;\;\; \int_\Omega v_\eps(\cdot, t)^p \leq K_2(T, p) \stext{ and therefore } \int_0^T\int_\Omega v_\eps^p \leq T K_2(T, p)
	\] 
	for all $t\in[0,T]$ and $\eps \in (0,1)$, which by Vitali's theorem result in
	\begin{align*}
		&u_{\eps_j} \rightarrow  u && \text{ in } L^p_\loc(\overline{\Omega}\times[0,\infty)) \text{ for } p \in [1,2), \\
		&v_{\eps_j}(\cdot, t) \rightarrow  v(\cdot, t) &&\text{ in } L^p(\Omega) \text{ for all } p \geq 1 \text{ and a.e.\ } t > 0 \text{ as well as } \\
		&v_{\eps_j} \rightarrow  v && \text{ in } L^p_\loc(\overline{\Omega}\times[0,\infty)) \text{ for all } p \geq 1
	\end{align*} 
	as $j \rightarrow \infty$.
	\\[0.5em]
	Because of the $L^p_\loc(\overline{\Omega}\times[0,\infty))$ convergence of the sequence $(u_{\eps_j})_{j\in\N}$ for $p \in [1,2)$, one last set of successive subsequences extractions gives us our final desired convergence property in (\ref{eq:basic_convergence_props}), namely that the sequences $(u_{\eps_j}(\cdot, t))_{j\in\N}$ converge to $u(\cdot, t)$ in $L^p(\Omega)$ for a.e.\  $t > 0$ and $p \in [1,2)$.
	\\[0.5em]
	The almost everywhere nonnegativity of $u$ and lower bound for $v$, which ensures the regularity property $v^{-1} \in L_\loc^\infty(\overline{\Omega}\times[0,\infty))$ from  (\ref{wsol:regularity}), are inherited from the approximate solutions due to the almost everywhere pointwise convergence proven above. While most of the other regularity properties in (\ref{wsol:regularity}) are then already directly ensured by the convergence properties considered in  (\ref{eq:basic_convergence_props}), the remaining $L_\loc^\infty([0,\infty);L^p(\Omega))$ type regularity properties follow because of the convergence of the norms $\|u_{\eps_j}(\cdot, t)\|_\L{1}$ and $\|v_{\eps_j}(\cdot, t)\|_\L{p}$ towards $\|u(\cdot, t)\|_\L{1}$ and $\|v(\cdot, t)\|_\L{p}$ for almost every $t > 0$ and  $p\in[1,\infty)$ ensured by (\ref{eq:basic_convergence_props}) combined with already established boundedness properties in \Cref{lemma:l1_l2_props} and \Cref{lemma:v_bounds} for the approximate solutions.
\end{proof}
\section{An additional convergence property for $(\grad v_{\eps_j})_{j\in\N}$}
\label{section:grad_v_convergence}
While we already established a lot of convergence properties in the lemma above, we will still need one more critical strong convergence property for the sequence $(\grad v_{\eps_j})_{j\in\N}$ to handle the taxis-induced terms in (\ref{wsol:ln_u_inequality}). To derive said property, we follow an approach that can be found, for instance, in \cite[Lemma 4.4]{MR3859449} or \cite[Lemma 8.2]{MR3383312}, as both of these papers deal with very similar solution concepts and therefore also have very similar needs in terms of convergence properties. 
\\[0.5em]
The first step towards the convergence property proven in \Cref{lemma:grad_v_convergence} later in this section is to argue that $v$ in fact already satisfies (\ref{wsol:v_inequality}). We do this by using the convergence properties in \Cref{lemma:subsequence_extraction} to show that (\ref{wsol:v_inequality}) directly translates from the approximate solutions to $v$ as follows:
\begin{lemma}\label{lemma:v_is_weak}
	Let $v$ be as in \Cref{lemma:subsequence_extraction}. Then $v$ satisfies (\ref{wsol:v_inequality}) for the same functions $\phi$ as in \Cref{definition:weak_solution}. 
\end{lemma}
\begin{proof}
	We first fix a test function $\phi \in \medcap_{p\geq 1} L^\infty_\loc((0,\infty); \L{p}) \cap L^2((0,\infty);W^{1,2}(\Omega))$ with $\phi_t \in L^2(\Omega\times[0,\infty))$. It is then easily checked by partial integration that each $v_\eps$ satisfies (\ref{wsol:v_inequality}) with said $\phi$ and as such we need now only further check that the equality survives the limit process $\eps=\eps_j \searrow 0$. For most of the terms in (\ref{wsol:v_inequality}), this is immediately obvious from the convergence properties seen in \Cref{lemma:subsequence_extraction} and therefore we will only give the argument for the $\int_\Omega u_\eps v_\eps \phi$ term as the $u_\eps v_\eps$ growth term is generally the primary source of complications in the second equation of (\ref{approx_problem}).
	\\[0.5em]
	For this, let now $T > 0$ be such that $\supp(\phi) \subseteq \overline{\Omega}\times[0,T]$ and then observe that
	\begin{align*}
		&\left| \int_0^\infty\int_\Omega u_\eps v_\eps \phi - \int_0^\infty\int_\Omega u v \phi \right| \\
		\leq& \int_0^T\int_\Omega |v_\eps||u_\eps - u||\phi| + \int_0^T\int_\Omega|v_\eps - v||u||\phi| \\
		\leq& \|v_\eps\|_{L^5(\Omega\times(0,T))} \|u_\eps - u\|_{L^\frac{5}{3}(\Omega\times(0,T))} \|\phi\|_{L^5(\Omega\times(0,T))} + \|v_\eps - v\|_{L^5(\Omega\times(0,T))} \|u\|_{L^\frac{5}{3}(\Omega\times(0,T))} \|\phi\|_{L^5(\Omega\times(0,T))}
	\end{align*}
	for all $\eps \in (0,1)$.
	Due to the fact that $L^\infty((0, T);\L{5}) \hookrightarrow L^5(\Omega\times(0,T))$ and the convergence properties laid out in \Cref{lemma:subsequence_extraction}, this then implies 
	\[
		\int_\Omega u_\eps v_\eps \phi \rightarrow \int_\Omega u v \phi 
	\]
	as $\eps = \eps_j \searrow 0$. This completes the proof.
\end{proof}
\noindent As the convergence properties in \Cref{lemma:subsequence_extraction} for the sequence $(v_{\eps_j})_{j\in \N}$ already provide us with the estimate \[
	\int_0^T\int_\Omega |\grad v|^2 \leq \liminf_{\eps=\eps_j \searrow 0} \int_0^T\int_\Omega |\grad v_\eps|^2,
\]
we now derive an important inequality from (\ref{wsol:v_inequality}) that will help us gain the corresponding estimate from below. To do this, the natural approach would be setting $\phi = v$ in (\ref{wsol:v_inequality}), which is not possible due to insufficient time regularity of $v$. Therefore, we have to approximate $v$ with time averaged versions of itself and then use those as test functions $\phi$. While using this approximation does not allow us to recover (\ref{wsol:v_inequality}) with $\phi = v$ exactly, it still provides us with an inequality version that is sufficient for our purposes.
\\[0.5em]
As this approach is very similar to the one used in \cite[Lemma 4.4]{MR3859449} or \cite[Lemma 8.2]{MR3383312} for a corresponding inequality, we will only give the following argument in brief:
\begin{lemma} \label{lemma:lower_v_bound}
	Let $v$ be as in \Cref{lemma:subsequence_extraction}.
    There exists a null set $N \subseteq (0,\infty)$ such that
	\begin{equation}
		\frac{1}{2}\int_\Omega v^2(\cdot, T) - \frac{1}{2}\int_\Omega v_0^2 + \int_0^T \int_\Omega |\grad v^2| \geq  
	     \int_0^T \int_\Omega uv^2 - \int_0^T \int_\Omega v^2 \label{eq:v_ineqality}
	\end{equation}
	for all $T \in (0,\infty)\setminus N$. 
\end{lemma}
\begin{proof}
    As in the references, we start by first fixing a null set $N \subseteq (0, \infty)$ such that each $T\in(0,\infty)\setminus N$ is a Lebesgue point of the map
    \[
	    [0,\infty) \rightarrow [0,\infty), \;\;\;\; t \mapsto \int_\Omega v^2(x,t)\d x
    \] and we then fix one such $T$. While we ourselves will not reiterate this argument from the references, this property of $T$ is mainly used there to ensure that 
    \[
	    \frac{1}{\delta}\int_{T}^{T+\delta} \int_\Omega v^2(x,t) \d x \d t \rightarrow \int_\Omega v^2(x,T) \d x
	\] for $\delta \rightarrow 0$.
    \\[0.5em]
    Because $v$ does not have all the necessary regularity properties to be used as a test function in (\ref{wsol:v_inequality}), which is what we want to essential do, due to us not knowing much about its time derivative, we then construct a time averaged version of $v$ with regularized initial data to take its place as follows:
    \\[0.5em]
	Let first $(v_{0k})_{k\in\N} \subseteq C^1(\overline{\Omega})$ be such that $v_{0k} \rightarrow v_0$ in $L^2(\Omega)$ as $k \rightarrow \infty$ due to density. Let then $\zeta_\delta$ be a cut-off function on $[0,\infty)$ such that $\zeta_\delta \equiv 1$ on $[0,T]$ and $\zeta_\delta \equiv 0$ on $[T+\delta, \infty)$ for $\delta \in (0,1)$ constructed in the same way as in the references. Further let
	\[
		\tilde{v}_k(x,t) \defs \begin{cases}
		v(x,t), \;\;\;\;\;\;& (x,t) \in \Omega \times (0,\infty), \\
		v_{0k}(x), &(x,t) \in \Omega \times (-1, 0]
		\end{cases}
	\]
	and then let $\phi(x,t) \defs \phi_{h,\delta, k}(x,t) \defs \zeta_\delta(t) (A_h\tilde{v}_k)(x,t)$ for all $(x,t) \in \Omega\times(0,\infty)$ with
	\[
		(A_h\tilde{v}_k)(x,t) \defs \frac{1}{h}\int_{t-h}^t \tilde{v}_k(x,s) \d s
	\]
	for all $(x,t) \in \Omega \times (0,\infty)$ and $\delta,h \in (0,1), k\in \N$. Similar to the references, it is then easy to show that the regularity properties of $v$ are enough to ensure that $\phi$ is a valid test function for (\ref{wsol:v_inequality}). We are therefore allowed to apply it to said equality with the aim to gain (\ref{eq:v_ineqality}) after a number of limit processes for the parameters $\delta, h$ and $k$. Most of the resulting integrals remain the same as in the references and converge or can be estimated in a similar fashion due to the regularity of $v$ and the fact that it implies that
	\[
		(A_h \tilde{v}_k)  \rightharpoonup \tilde{v}_k \stext{ in all }L^p(\Omega\times(0,T)) \text{ for } p \in (1,\infty) \text{ as } h \searrow 0
	\]
	and
	\[
		\grad (A_h \tilde{v}_k) = (A_h \grad\tilde{v}_k)  \rightharpoonup \grad \tilde{v}_k \stext{ in }L^2(\Omega\times(0,T)) \text{ as } h \searrow 0
	\]	
	because of \cite[Lemma A.2]{MR3383312} for all $k\in\N$. It is the estimates for these integrals, which we will not discuss here in more detail, that lead to us only deriving (\ref{eq:v_ineqality}) as an inequality as opposed to the equality one would expect for $\phi = v$. For the full details concerning this, see e.g.\ \cite[Lemma 4.4]{MR3859449} or \cite[Lemma 8.2]{MR3383312}.
	\\[0.5em]
	We therefore will only take a closer look at the two integrals new to our setting: As we know due to the already established regularity properties of $u$ and $v$ that $uv \in L_\loc^{p}(\overline{\Omega}\times[0,\infty))$ for all $p \in [1,2)$ due to the Hölder inequality, we immediately see that
	\[
		\int_0^\infty\int_\Omega \zeta_\delta(t) u(x,t)v(x,t) (A_h \tilde{v}_k)(x,t) \d x \d t \rightarrow \int_0^\infty \int_\Omega \zeta_\delta(t) u(x,t)v^2(x,t) \d x \d t \rightarrow \int_0^T \int_\Omega u(x,t)v^2(x,t) \d x \d t
	\]
	and similarly that
	\[
		\int_0^\infty\int_\Omega \zeta_\delta(t) v(x,t) (A_h \tilde{v}_k)(x,t)  \d x \d t \rightarrow \int_0^\infty \int_\Omega \zeta_\delta(t) v^2(x,t) \d x \d t \rightarrow \int_0^T \int_\Omega v^2(x,t)  \d x \d t
	\]
	as first $h\searrow 0$ and then $\delta \searrow 0$ for all $k\in\N$. Note hereby that the $\delta \searrow 0$ limit process works due to the dominated convergence theorem.
\end{proof}
\noindent Given this inequality, we can now prove the following important convergence property:
\begin{lemma} \label{lemma:grad_v_convergence}
	Let the function $v$ and sequence $(\eps_j)_{j\in\N}$ be as in \Cref{lemma:subsequence_extraction}. Then
	\[
		\grad v_\eps \rightarrow \grad v \stext{ as } \eps = \eps_j \searrow 0 
	\]
	in $L_\loc^2(\overline{\Omega}\times[0,\infty))$.
\end{lemma}
\begin{proof}
	Fix $T \in (0,\infty)\setminus N$ with $N$ as in \Cref{lemma:lower_v_bound}. As the already established convergence properties in \Cref{lemma:subsequence_extraction} for $v$ give us that
	\[
	\int_0^T\int_\Omega |\grad v|^2 \leq \liminf_{\eps=\eps_j \searrow 0} \int_0^T\int_\Omega |\grad v_\eps|^2,
	\]
	it is sufficient to prove a similar estimate from below.
	\\[0.5em]
	As a preparation for this, let us now first observe that
	\begin{align*}
		&\left| \; \int_0^T\int_\Omega u v^2 - \int_0^T\int_\Omega u_\eps v_\eps^2 \; \right| \\
		&\leq \int_0^T\int_\Omega |u - u_\eps| v^2 + \int_0^T\int_\Omega u_\eps |v - v_\eps| v + \int_0^T\int_\Omega u_\eps v_\eps |v - v_\eps| \\
		&\leq \|u - u_\eps\|_{L^\frac{4}{3}(\Omega\times(0,T))}\|v\|^2_{L^8(\Omega\times(0,T))} + \|u_\eps\|_{L^\frac{5}{3}(\Omega\times(0,T))}\|v - v_\eps\|_{L^{5}(\Omega\times(0,T))}\|v\|_{L^{5}(\Omega\times(0,T))} \\
		&\;\;\;\;+ \|u_\eps\|_{L^\frac{5}{3}(\Omega\times(0,T))}\|v_\eps\|_{L^{5}(\Omega\times(0,T))}\|v - v_\eps\|_{L^{5}(\Omega\times(0,T))} 
	\end{align*}
	for all $\eps \in (0,1)$. Due to the boundedness and convergence properties in \Cref{lemma:l1_l2_props}, \Cref{lemma:v_bounds} and \Cref{lemma:subsequence_extraction}, this then implies that
	\[
		  \int_0^T\int_\Omega u_\eps v_\eps^2 \rightarrow \int_0^T\int_\Omega u v^2 
	\]
	as $\eps = \eps_j \searrow 0$.
	\\[0.5em]
	Using this convergence property as well as the properties laid out in \Cref{lemma:subsequence_extraction} in combination with \Cref{lemma:lower_v_bound}, we directly see that
	\begin{align*}
		\int_0^T\int_\Omega |\grad v|^2
		&\geq -\frac{1}{2}\int_\Omega v^2(\cdot, T) + \frac{1}{2}\int_\Omega v_0^2  +
		\int_0^T \int_\Omega uv^2 - \int_0^T\int_\Omega v^2 \\
		&= \lim_{\eps = \eps_j \searrow 0} \left\{ -\frac{1}{2}\int_\Omega v_\eps^2(\cdot, T) + \frac{1}{2}\int_\Omega v_0^2  +
			\int_0^T \int_\Omega u_\eps v_\eps^2 - \int_0^T\int_\Omega v_\eps^2 
	    \right\} = \lim_{\eps = \eps_j \searrow 0}   \int_0^T\int_\Omega |\grad v_\eps|^2.
	\end{align*}
	This completes the proof.
\end{proof}

\section{Proof of \Cref{theorem:main}}
\noindent Having now assembled all the necessary convergence properties for the sequences $(u_{\eps_j})_{j\in\N}$, $(v_{\eps_j})_{j\in\N}$ and even already some necessary properties for the limit functions and solution candidates $u$ and $v$, we can now begin the proof of our central result. 
\begin{proof}[Proof of \Cref{theorem:main}]
	Let the functions $u,v$ and sequence $(\eps_j)_{j\in\N}$ be as in \Cref{lemma:subsequence_extraction}.
	\\[0.5em]
	As the properties (\ref{wsol:regularity}) and (\ref{wsol:v_inequality}) for $u$ and $v$ have already been established in \Cref{lemma:subsequence_extraction} and \Cref{lemma:v_is_weak} respectively, we only need to still prove the inequalities (\ref{wsol:mass_property}) and (\ref{wsol:ln_u_inequality}) for $u$ here.
	\\[0.5em]
	We start with (\ref{wsol:mass_property}). By just integrating the first equation the approximate solutions $u_\eps$ solve, we directly see that
	\begin{equation}
		\int_\Omega u_\eps(\cdot, T) - \int_\Omega u_0 = -\int_0^T \int_\Omega u_\eps v_\eps + \rho\int_0^T \int_\Omega u_\eps - \mu \int_0^T \int_\Omega u_\eps^2 \;\;\;\; \text{ for all } \eps \in (0,1) \text{ and } T > 0. \label{eq:approx_mass_property}
	\end{equation}
	Apart from $\int_0^T \int_\Omega u_\eps^2$, all of the terms above converge to their equivalent without $\eps$ due to \Cref{lemma:subsequence_extraction}, while we only get that
	\[
		\int_0^T \int_\Omega u^2  \leq \liminf_{\eps = \eps_j \searrow 0} \int_0^T \int_\Omega u_\eps^2
	\]
	for the remaining term due to weak convergence. Taking the limes superior on both sides of (\ref{eq:approx_mass_property}) then immediately yields (\ref{wsol:mass_property}).
	\\[0.5em]
	We now fix a nonnegative $\phi \in C_0^\infty(\overline{\Omega}\times[0,\infty))$ with $\grad \phi \cdot \nu = 0$ on $\partial \Omega\times(0,\infty)$. Similar to the above, testing the first equation in (\ref{approx_problem}) with $\frac{\phi}{u_\eps + 1}$ and partial integration yields (\ref{wsol:ln_u_inequality}) with equality and some slightly different taxis terms due to the cut-off function $\eta_\eps$ for the approximate solutions $u_\eps$.
	Then apart from 
	\[
		\int_0^\infty \int_\Omega |\grad \ln(u_\eps+1)|^2\phi,
	\]
	all of the remaining integral terms are convergent to their counterparts without $\eps$ and without the cut-off function $\eta_\eps$ as we will now briefly illustrate:
	\\[0.5em]
	Due to the $L^2_\loc(\overline{\Omega}\times[0,\infty))$ convergence of the sequence $(\ln(u_{\eps_j} + 1))_{j\in\N} $, we immediately gain that
	\[
		\int_0^\infty\int_\Omega \ln(u_\eps + 1) \phi_t \rightarrow \int_0^\infty\int_\Omega \ln(u + 1) \phi_t \stext{and}
		\int_0^\infty\int_\Omega \ln(u_\eps + 1) \laplace\phi \rightarrow \int_0^\infty\int_\Omega \ln(u + 1) \laplace \phi
	\]
	as $\eps = \eps_j \searrow 0$. To handle the taxis-induced terms, let us first observe that
	\begin{align*}
		&\left\| \; \frac{\eta_\eps(u_\eps)u_\eps}{v_\eps(u_\eps+1)} \grad v_\eps - \frac{u}{v(u+1)} \grad v \; \right\|_{L^{2}(\Omega\times(0,T))} \\
		\leq& \left\| \frac{\eta_\eps(u_\eps) u_\eps}{v_\eps(u_\eps+1)} \right\|_{L^{\infty}(\Omega\times(0,T))}  \left\| \grad v_\eps - \grad v \right\|_{L^{2}(\Omega\times(0,T))}  + \int_0^T\int_\Omega |\grad v|^2 \left(  \frac{\eta_\eps(u_\eps)u_\eps}{v_\eps(u_\eps+1)} - \frac{u}{v(u+1)}\right)^2
	\end{align*}
	for each $T > 0$ by introducing a zero. We then further note that, for each $T > 0$, \Cref{lemma:approx_exist} and \Cref{lemma:subsequence_extraction} give us that $\frac{\eta_\eps(u_\eps)u_\eps}{v_\eps(u_\eps+1)}$ and $\frac{u}{v(u+1)}$ are uniformly bounded on $\Omega\times(0,T)$ independent of $\eps$ and that $\grad v\in L^2(\Omega\times(0,T))$, which implies that the first integral term converges to zero as $\eps = \eps_j \searrow 0$ due to \Cref{lemma:grad_v_convergence} and the second integral term converges to zero as $\eps = \eps_j \searrow 0$ due to the pointwise convergence proven in \Cref{lemma:subsequence_extraction} combined with the dominated convergence theorem. Thus,
	\[
		 \frac{\eta_\eps(u_\eps)u_\eps}{v_\eps(u_\eps+1)}\grad v_\eps \rightarrow \frac{u}{v(u+1)}\grad v \stext{in}L^2_\loc(\overline{\Omega}\times[0,\infty)) \text{ as } \eps = \eps_j \searrow 0.
	\]
	If we then combine this with the weak convergence properties of the sequence $(\grad\ln(u_{\eps_j} + 1))_{j\in\N}$ in $L^2_\loc(\overline{\Omega}\times[0,\infty))$ from \Cref{lemma:subsequence_extraction}, we directly gain that
	\[
		\int_0^\infty \int_\Omega \frac{\eta_\eps(u_\eps)u_\eps}{v_\eps(u_\eps+1)} (\grad \ln(u_\eps + 1) \cdot \grad v_\eps) \phi \rightarrow  \int_0^\infty \int_\Omega \frac{u}{v(u+1)} (\grad \ln(u + 1) \cdot \grad v) \phi
	\]
	and
	\[
		\int_0^\infty \int_\Omega \frac{\eta_\eps(u_\eps)u_\eps}{v_\eps(u_\eps+1)} \grad v_\eps \cdot \grad \phi \rightarrow  \int_0^\infty \int_\Omega \frac{u}{v(u+1)} \grad v \cdot \grad \phi
	\]
	as $\eps = \eps_j \searrow 0$. As for the convergence of the remaining three relevant integral terms
	\[
		\int_0^\infty\int_\Omega \frac{u_\eps v_\eps}{u_\eps+1} \phi, \;\;\;\; \int_0^\infty\int_\Omega \frac{u_\eps}{u_\eps+1} \phi \stext{and} \int_0^\infty\int_\Omega \frac{u_\eps^2}{u_\eps+1} \phi
	\]
	towards their counterparts without $\eps$, the above argument can essentially be reused as they also feature the product of a pointwise convergent and uniformly bounded sequence of functions, which in this case is always $\frac{u_\eps\phi}{u_\eps + 1}$, and a sequence of functions converging in an appropriate $L^p_\loc(\overline{\Omega}\times[0,\infty))$ as their integrand.
	\\[0.5em]
	As in reference \cite{MR3383312}, for the remaining term we at least have the property
	\[
		\int_0^\infty \int_\Omega |\grad \ln(u+1)|^2\phi \leq \liminf_{\eps=\eps_j \searrow 0} \int_0^\infty \int_\Omega |\grad \ln(u_\eps+1)|^2\phi
	\]
	due to the weak convergence proven in \Cref{lemma:subsequence_extraction}. This then gives us (\ref{wsol:ln_u_inequality}) after taking the limes inferior of both sides of the approximated variant of (\ref{wsol:ln_u_inequality}). As such, $(u, v)$ is in fact a generalized solution in the sense of \Cref{definition:weak_solution} and the proof is complete.
\end{proof}

\section{Existence of classical solutions to the altered system (\ref{weaker_problem})}
\label{section:weakend_case}
As already mentioned in the introduction, we will devote this section to proving \Cref{prop:weaker_system_stronger_solution}, which is concerned with the global classical solvability of the altered system (\ref{weaker_problem}) featuring a stronger logistic source term. We do this to illustrate how close the interplay between the logistic term in the first equation and the growth term in the second equation is to immediately giving us global classical solvability for (\ref{problem}) while just about not being sufficient in our opinion.
\\[0.5em]
As our first step for this, we can use similar standard arguments as used in \Cref{lemma:approx_exist} to derive the following local existence result and blow-up criterion for (\ref{weaker_problem}):
\begin{lemma} \label{lemma:weaker_case_exist}
	There exists a maximal constant $\tmax \in (0,\infty]$ and functions $u,v \in C^{2,1}(\overline{\Omega}\times(0,\tmax))\cap C^0(\overline{\Omega}\times[0,\tmax))$ with $u$ nonegative and $v$ positive such that $(u,v)$ is the unique classical solution of the system (\ref{weaker_problem}) with (\ref{boundary_conditions})--(\ref{intial_data_regularity}) on $\overline{\Omega}\times[0,\tmax)$. Further, the solution $(u,v)$ adheres to the following blow-up criterion:
	\begin{equation}
		\label{eq:weak_blowup}
		\text{If } \tmax < \infty, 
		\text{ then } \limsup_{t\nearrow \tmax} \left\{ \|u(\cdot, t)\|_\L{\infty} + \|v(\cdot, t)\|_{W^{1,q}(\Omega)} \right\} = \infty \text{ or } \liminf_{t\nearrow \tmax} \inf_{x\in\Omega} v(x,t) = 0.
	\end{equation} 
	Here, $q$ is some real number in $(2,2+\gamma)$.
\end{lemma}
\noindent We now fix appropriate initial data $(u_0, v_0)$, a unique solution $(u,v)$ on $\overline{\Omega}\times[0,\tmax)$ corresponding to said initial data and $\tmax \in (0,\infty]$, $q\in(2,2+\gamma)$ according to \Cref{lemma:weaker_case_exist}.
\\[0.5em]
One additional result directly reusable from the existence theory in \Cref{lemma:approx_exist} is that
\begin{equation}
	v(\cdot, t) \geq e^{-t}\inf_{x\in\Omega}v_0(x) \;\;\;\; \text{ for all } t\in[0,\tmax) \label{eq:weak_lower_v_bound}
\end{equation}
due to semigroup methods, which immediately prevents one of the possible blow-up scenarios.
\\[0.5em]
As many techniques to derive a priori estimates for $(u,v)$ translate directly from \Cref{section:apriori} due to the changes above in a sense only working in our favor, we will now only briefly revisit some of the foundational results from said section and translate these to $(u,v)$. 
\begin{lemma} \label{lemma:weak_baseline_estimates}
	There exists a constant $C_1 > 0$ such that
	\[
		\|u(\cdot, t)\|_\L{1} \leq C_1,\;\; \|v(\cdot, t)\|_\L{1} \leq C_1 \;\;\;\; \text{ for all } t\in[0,\tmax)
	\]
	and, if $\tmax < \infty$, there exists a constant $C_2 > 0$ such that
	\[
		\int_0^\tmax \int_\Omega u^{2} \leq C_2, \;\;\;\; \int_0^\tmax \int_\Omega u^{2+\gamma} \leq C_2.
	\]
\end{lemma}
\begin{proof}
	Similar to the proof of \Cref{lemma:l1_l2_props}, we can gain these bounds by adding the first and second equation and integrating to see that
	\[
		\frac{\d }{\d t} \int_\Omega (u + v) \leq -\int_\Omega (u + v) + K_1 \;\;\;\; \text{ for all } t \in [0,\tmax)
	\]
	with $K_1 \defs [ \frac{1+|\rho|}{2+\gamma} ]^{2+\gamma}\frac{1+\gamma}{\mu^{1+\gamma}}|\Omega|$ due to Young's inequality. Further if $\tmax < \infty$ and we just integrate the first equation in (\ref{weaker_problem}), we gain that
	\[
		\int_0^\tmax \int_\Omega u^{2+\gamma} \leq \frac{1}{\mu}\int_\Omega u_0 + \frac{|\rho|}{\mu}\int_0^\tmax \int_\Omega u.
	\]
	As in \Cref{lemma:l1_l2_props}, these inequalities directly imply most of our results while the last remaining bound follows due to the Hölder inequality and the fact that $2 < 2 + \gamma$.
\end{proof}
\noindent Because the second equation in the approximated system (\ref{approx_problem}) and the altered system (\ref{weaker_problem}) are the same, \Cref{lemma:v_bounds} translates almost verbatim.
\begin{lemma}\label{lemma:weak_vp_bounds}
	If $\tmax < \infty$, there exists a constant $C(p) > 0$ such that
	\[
		\|v(\cdot, t)\|_\L{p} \leq C(p) 
	\]
	for all $t \in [0,\tmax)$. 
\end{lemma}
\noindent While it might only seem like a slight improvement, the leeway afforded to us by $\gamma > 0$ then allows us to nonetheless achieve a critical $L^\infty$ bound for $v$ and an additional bound for the gradient of $v$, which both eluded us in the case discussed in the previous sections. It is both of these results that will ultimately prove to be the key to the existence of global classical solutions for this case.
\begin{lemma} \label{lemma:weak_higher_v_bounds}
	If $\tmax < \infty$, there exists a constant $C > 0$ such that
	\[
		\|v(\cdot, t)\|_\L{\infty} \leq C 
	\]
	and
	\[
		\|\grad v(\cdot, t)\|_\L{q} \leq C
	\]
	for all $t \in [0,\tmax)$.
\end{lemma} 
\begin{proof} 
By using the mild solution representation of $v$ (relative to the semigroup $e^{t(\laplace - 1)}$) and well-known smoothness estimates for said semigroup, we see that
\begin{align*}
	\|v(\cdot, t)\|_\L{\infty} &\leq \|v_0\|_\L{\infty} + K_1\int_0^t (t-s)^{-\frac{1}{2}}e^{-(t-s)} \|u v\|_\L{2} \d s \\
	&\leq  \|v_0\|_\L{\infty} + K_1|\Omega|^{\frac{q-2}{2q}}\int_0^t (t-s)^{-\frac{1}{2}}e^{-(t-s)} \|u v\|_\L{q} \d s 
\end{align*}
and
\[
	\|\grad v(\cdot, t)\|_\L{q} \leq K_1\|\grad v_0\|_\L{q} + K_1\int_0^t (t-s)^{-\frac{1}{2}}e^{-(t-s)} \|u v\|_\L{q} \d s	
\]
for all $t \in [0,\tmax)$ and some constant $K_1 > 0$. Using the Hölder and Young inequalities, we can now further estimate the critical integral term in both of the above inequalities as
\begin{align*}
	\int_0^t (t-s)^{-\frac{1}{2}}e^{-(t-s)} \|u v\|_\L{q} \leq& \int_0^t (t-s)^{-\frac{1}{2}}e^{-(t-s)} \|u\|_\L{2+\gamma}\|v\|_\L{p} \d s \\
	\leq& \frac{1}{2+\gamma}\int_0^t \int_\Omega u^{2 + \gamma} \d s  + \frac{1}{r} \int_0^t (t-s)^{-\frac{r}{2}}e^{-r(t-s)} \|v\|^r_\L{p} \d s \label{eq:critical_integral_term} \numberthis
\end{align*}
with $p \defs \frac{q(2 + \gamma)}{2 + \gamma - q} \in (q,\infty)$ because $q\in(2,2+\gamma)$ and $r \defs \frac{2 + \gamma}{1 + \gamma} \in (1,2)$ for all $t \in [0,\tmax)$. Due to \Cref{lemma:weak_baseline_estimates}, \Cref{lemma:weak_vp_bounds} and the fact that $\frac{r}{2} < 1$, the remaining integrals in (\ref{eq:critical_integral_term}) are uniformly bounded for all $t \in [0,\tmax)$, which directly implies our desired results.
\end{proof}
\noindent By similar semigroup methods, we can now gain a corresponding result for the first solution component $u$:
\begin{lemma}\label{lemma:weak_linfty_u}
If $\tmax < \infty$, there exists $C > 0$ such that
	\[
		\|u(\cdot, t)\|_\L{\infty} \leq C
	\]
	for all $t\in[0,\tmax)$.
\end{lemma}
\begin{proof}
Due to the fact that there exists $K_1 > 0$ such that $\rho y - \mu y^{2 + \gamma} \leq K_1$ for all $y \geq 0$, we can estimate the mild solution representation of $u$ (relative to the semigroup $e^{t\laplace}$) as follows:
\[
	u(\cdot, t) \leq \|u_0\|_\L{\infty} + \chi \int_0^t e^{(t-s)\laplace} \div (\tfrac{u}{v} \grad v) \d s + \int_0^t K_1  \d s \leq K_2 + \chi\int_0^t e^{(t-s)\laplace} \div (\tfrac{u}{v} \grad v)   \d s
\]	
with $K_2 \defs \|u_0\|_\L{\infty} + \tmax K_1$ for all $t\in[0,\tmax)$. Now fix $p\in (2,q)$. By well-known semigroup smoothness estimates and the Hölder inequality, we can then improve the above to
\begin{align*}
	\|u(\cdot, t)\|_\L{\infty} &\leq K_2 + \chi K_3 \int_0^t (t-s)^{-\frac{1}{2} - \frac{1}{p}} \|\tfrac{u}{v} \grad v\|_\L{p}  \d s \\
	&\leq K_2 + K_4\int_0^t (t-s)^{-\frac{1}{2} - \frac{1}{p}} \|u\|_\L{r} \|\grad v\|_\L{q}  \d s \\
	&\leq K_2 + K_4\int_0^t (t-s)^{-\frac{1}{2} - \frac{1}{p}} \|u\|^\alpha_\L{1} \|u\|^{1-\alpha}_\L{\infty} \|\grad v\|_\L{q}  \d s
\end{align*}
with some constant $K_3 > 0$, $K_4 \defs \chi K_3 e^{\tmax}(\inf_{x\in\Omega} v_0(x))^{-1},\; r \defs \frac{pq}{q-p} \in (p,\infty)$ and $\alpha \defs \frac{1}{r}$ for all $t \in [0,\tmax)$. If we now define $M_T \defs \|u\|_{L^\infty(\Omega\times[0,T])} < \infty$ for every $T \in [0,\tmax)$, the above inequality allows us to derive that
\[
	M_T \leq K_2 + K_5 M_T^{1-\alpha}
\]
for some $K_5 > 0$, which is independent of $T$, due to \Cref{lemma:weak_baseline_estimates}, \Cref{lemma:weak_higher_v_bounds} and the fact that $\frac{1}{2} + \frac{1}{q} < 1$. This implies that there exists a constant $K_6 > 0$ such that $M_T \leq K_6$ for all $T\in [0,\tmax)$. This completes the proof.
\end{proof}
\noindent
\Cref{lemma:weak_higher_v_bounds} and \Cref{lemma:weak_linfty_u} have now shown that the remaining blow-up scenarios in (\ref{eq:weak_blowup}) are impossible as well and therefore we can now prove \Cref{prop:weaker_system_stronger_solution} as follows:
\begin{proof}[Proof of \Cref{prop:weaker_system_stronger_solution}]
	Assume $\tmax < \infty$. Then by (\ref{eq:weak_lower_v_bound}), we directly gain that
	\begin{equation}
		\inf_{x\in\Omega}v(x,t) \geq e^{-\tmax}\inf_{x\in\Omega} v_0(x) > 0  \;\;\;\; \text{ for all } t\in [0,\tmax). \label{eq:no_blowup1}
	\end{equation}
	Further due to \Cref{lemma:weak_higher_v_bounds} and \Cref{lemma:weak_linfty_u}, we can gain $K_1 > 0$ such that
	\begin{equation}
		\|v(\cdot,t)\|_{W^{1,q}(\Omega)} \leq K_1, \;\;\;\; \|u(\cdot, t)\|_\L{\infty} \leq K_1 \;\;\;\; \text{ for all } t\in [0,\tmax). \label{eq:no_blowup2}
	\end{equation}
	Together (\ref{eq:no_blowup1}) and (\ref{eq:no_blowup2}) contradict the blow-up criterion (\ref{eq:weak_blowup}) and therefore we must have $\tmax = \infty$, which completes the proof.
\end{proof}

\section*{Acknowledgment} The author acknowledges support of the \emph{Deutsche Forschungsgemeinschaft} in the context of the project \emph{Emergence of structures and advantages in cross-diffusion systems}, project number 411007140.

\end{document}